\documentclass[11pt]{amsart}
\usepackage{amsfonts}
\usepackage{amsmath,amsthm}
\usepackage{latexsym}
\usepackage{array}
\usepackage{amssymb}
\usepackage{enumerate}

\DeclareFontFamily{U}{wncy}{}
\DeclareFontShape{U}{wncy}{m}{n}{<->wncyr10}{}
\DeclareSymbolFont{mcy}{U}{wncy}{m}{n}
\DeclareMathSymbol{\Sh}{\mathord}{mcy}{"58}

\usepackage[english]{babel}
\usepackage{fdsymbol}
\usepackage{color}
\usepackage[latin1]{inputenc}

\numberwithin{equation}{section}
\numberwithin{figure}{section}

\newtheorem{thm}{Theorem}[section]
\newtheorem{cor}[thm]{Corollary}
\newtheorem{lem}[thm]{Lemma}
\newtheorem{prop}[thm]{Proposition}
\theoremstyle{definition}
\newtheorem{defn}[thm]{Definition}
\newtheorem{asmp}{Assumption}
\theoremstyle{remark}

\newcommand{\ds}{\displaystyle}

\newcommand{\R}{ \mathbb{R} }

\newcommand{\T}{ \mathbb{T} }

\newcommand{\cP}{ \mathcal{P} }

\newcommand{\cD}{ \mathcal{D} }

\newcommand{\cK}{ \mathcal{K} }

\newcommand{\Lc}{ \mathcal{L} }

\newcommand{\cM}{ \mathcal{M} }

\newcommand{\ga}{\gamma }

\renewcommand{\phi}{ \varphi }

\newcommand{\eps}{\varepsilon}

\newcommand{\de}{ \delta }
\newcommand{\al}{ \alpha }

\newcommand{\dist}{{\operatorname{dist}}}

\newcommand{\be}{\begin{equation}}
\newcommand{\ee}{\end{equation}}
\newcommand{\ben}{\begin{equation*}}
\newcommand{\een}{\end{equation*}}

\newcommand{\p}{ \partial}

\newcommand{\wt}{ \widetilde }

\newcommand{\lbl}{\label}
\newcommand{\non}{\nonumber}
\newcommand{\qu}{\quad}
\newcommand{\qmb}{\quad\mbox}

\newcommand{\Ga}{\Gamma}

\newcommand{\sm}{\setminus}

\newcommand{\ssk}{\smallskip}

\newcommand{\bsk}{\bigskip}

\DeclareMathOperator{\ord}{ord}

\title[]{Strict Lyapunov functions and energy decay in Hamiltonian chains with degenerate damping}

\author{Andrey Dymov}
\address{Steklov Mathematical Institute of Russian Academy of Sciences, Moscow 119991, Russia 
	\& National Research University Higher School of
	Economics, Moscow 119048, Russia
} 
\email{dymov@mi-ras.ru}

\author{Lev Lokutsievskiy}
\address{Steklov Mathematical Institute of Russian Academy of Sciences, Moscow 119991, Russia} 
\email{lion.lokut@gmail.com}

\author{Andrey Sarychev}
\address{Dipartimento di Matematica e Informatica U. Dini, Universit\`a di Firenze, Firenze 50134, Italia}
\email{andrey.sarychev@unifi.it}

\begin{document}

\begin{abstract}
	We consider a Hamiltonian chain of rotators (in general nonlinear) in which the first rotator is damped. Being motivated by problems of nonequilibrium statistical mechanics of crystals, 
	we construct a strict Lyapunov function that allows us to find a lower bound for the total energy dissipation rate when the energy and time are large. Our construction is explicit and its analysis is rather straightforward. We rely on a method going back to Matrosov, Malisoff and Mazenc, which we review in our paper. The method is rather universal and we show that it is  applicable to a chain of oscillators as well. 
\end{abstract}

\maketitle

\section{Introduction}

\subsection{Problem setting}

In this paper we consider a Hamiltonian chain of particles ~--- rotators or oscillators, in general nonlinear, and add dissipation to the first particle. More specifically, the chain of rotators is given by the Hamiltonian 
\be\lbl{H_rot}
H(p,q) = \sum_{j=1}^N \frac{p_j^2}{2} + \sum_{j=1}^{N-1} V_j(q_j-q_{j+1}),
\ee
where $N\ge 2,$ $p=(p_1,\dots, p_N)\in\R^N$ denote the moments of rotators and $q=(q_1,\dots, q_N)\in \T^N$ stand for their coordinates, each one living on the one-torus $\T\ni q_j$. The periodic potentials $V_j:\T\mapsto \R$ are assumed to be sufficiently smooth and satisfy non-degeneracy assumptions to be specified later. The Hamiltonian of the chain of oscillators is given by \eqref{H_osc}.
The equations of motion are 
\be
\lbl{eq_intro}\dot q_j = p_j, \qquad \dot p_j = -\p_{q_j} H - \de_{1j} p_1, 
\ee
where $\de_{1j}$ is the Kronecker symbol  and the dot stands for the time derivative.

Being motivated by problems of  nonequilibrium statistical mechanics (see Section~\ref{s:mot}),
we are   studying the total energy decay in this system when its energy is high. In the case of rotators, we  construct a suitable Lyapunov function for equation \eqref{eq_intro}, strict outside a compact set. Using it, we find a lower bound $C_N H^{-(2N-3)}$ for the energy dissipation rate  when the energy $H$ and time are sufficiently large.
 
More specifically,   our main result is Theorem~\ref{t:rotators}, in which we explicitly construct  a Lyapunov-type function $W=W(p,q)$, that behaves as $C_NH^{2N-1}$ when $H$ is large, satisfying
\[\dot W_t \le -C_{1,N} W_t^{1/(2N-1)} + C_{2,N}.\]
Here $W_t:=W(p,q)(t)$ and $(p,q)(t)$ is a solution to \eqref{H_rot}-\eqref{eq_intro},  while $C_N$, $C_{1,N}$ and $C_{2,N}$ are positive constants, depending on $N$.
Then, in Corollary~\ref{c:en_dec} we deduce that once the initial energy $H_0:=H(p,q)(0)$ is sufficiently large, the energy $H_t:=H(p,q)(t)$ satisfies
\be\lbl{en_dec_intro}
H_t \le H_0 - C_N\,t  H_0^{-(2N-3)}
\ee
 on time interval 
$H_0^{2N-\frac72}\ll t\ll H_0^{2N-2}$.
Proofs of Theorem~\ref{t:rotators} and Corollary~\ref{c:en_dec} are rather short and given in Section~\ref{s:rot} which can be read independently from the rest of the paper.

 In \cite{CEW} the authors gave numerical evidence that the lower bound \eqref{en_dec_intro} is optimal  in the sense that it is achieved at least for some potentials $V_j$ and initial conditions, but they studied it on shorter times $t\sim H_0^{(2N-3)/2}$; see below. This indicates that  the energy decay  in the chain of rotators can become extremely slow when the  energy $H$ is large. The reason for this phenomenon is that in the case when a $k$-th rotator has large energy but the energy of the $k-1$-st rotator is not so large, the argument $q_{k-1}-q_k$ of the interaction potential $V_{k-1}$ oscillates very fast. Then effective interaction of  these rotators almost vanishes due to an averaging-type effect, which leads to very slow energy transport through the chain. 

For the chain of oscillators we prove the existence of a strict Lyapunov function  for the high energy regime, but we do not find explicitly a lower bound for the energy decay rate. These computations can be  performed by analogy with those for the chain of rotators, and we hope to provide them elsewhere.

To get our result, we follow an approach going back to Matrosov, Malisoff and Mazenc, see \cite{MM}, allowing one under mild conditions to construct a strict Lyapunov function once a non-strict one is known (in our case the latter is the Hamiltonian $H$).
Roughly speaking, it suggests a way to construct functions $f_j(p,q)$ with $0\le j\le r$, $r\ge 1$,  satisfying the following properties.   The function $f_0$ has the form  $f_0=g(H)$ for a sufficiently large function $g$. The remaining functions $f_j$ are such that their Lie derivatives $Lf_j$ along the vector field from the r.h.s. of \eqref{eq_intro}  satisfy $L f_j|_{\{L f_{j-1}=0\}} \le 0$, outside the surface $\{L f_{j-1}=0\}$ the derivatives $L f_j$ are dominated by negative parts of the derivatives $L f_k$, $0\le k \le r$, and the surface $\{L f_j = 0\}$ has positive codimension in the surface $\{L f_{j-1}=0\}$.
Then the desired strict Lyapunov function is given by  $\sum_{j=0}^r  f_j$.

 Under additional restrictions the lower bound \eqref{en_dec_intro} was previously obtained in \cite{CEW}, but for times $t\sim H_0^{(2N-3)/2}$. 
 Namely, the authors considered  only a specific regime when the most part of the total energy is concentrated in a single rotator~\footnote{The authors obtained the lower bound $C_k H^{-(2k-3)}$ for the energy dissipation rate, where $k$ is a number of the "fast" rotator. In the case when $k=N$ this bound coincides with  ours.},
under the additional assumption that the potentials $V_j$ are trigonometric polynomials. 
  As we have mentioned above, they  presented numerical evidence of optimality of their lower bound.

In \cite{CEW} they used a completely different, technically  more complicated approach,  viewing the inverse momentum of the "fast" rotator as a small parameter   and using a KAM-like machinery. Namely, they performed a finite number of canonical transformations making explicit the discussed above decoupling of the "fast" rotator from its "slow" neighbours  and propagating the dissipation along the chain. Adaptation of this approach to the general case when the energy is distributed arbitrarily among the rotators  looks problematic, in particular  because of the resonances arising when both neighbouring rotators are  "fast".  On the other hand, in contrast to our method, approach used in \cite{CEW} allowed in \cite{CEP_3, CE_4} to establish the mixing property for chains of 3 and 4 rotators, coupled via their ends to thermal bath.  See a brief discussion in the next subsection and a more detailed one in Section~\ref{s:mixing}.

\subsection{Motivation: mixing in chains}
\lbl{s:mot}
Stability of Hamiltonian systems, sometimes with added dissipation, is a problem widely discussed in the literature, see e.g. excellent introduction to \cite{CEW} and references therein. 
However, our primary motivation comes from nonequilibrium statistical mechanics in which the chains of oscillators and rotators, coupled via the ends with heat baths, are classical models to study energy transport in crystals.
The heat baths are usually modelled by adding a dissipation (as in equation \eqref{eq_intro}) and stochastic perturbation to the equations of motion of the first and the last particles in the chain. 
In particular, one is interested in proving that the system approaches a stationary state when time goes to infinity, i.e. exhibits the {\it mixing} property.
Once the mixing property is established, the question of validity of the Fourier law arises ~--- a famous completely open problem \cite{Leb, LLP}. Despite many efforts by the community of mathematical physicists, some progress was achieved only in the presence of additional stochastic  perturbation (see e.g. \cite{BLL, LO, Dym16} and references therein) or assuming hyperbolicity of uncoupled dynamics \cite{DL, R}.

 For the moment of writing, the mixing property is proven only for the chains of rotators of length $N=3,4$ \cite{CEP_3, CE_4}. A source of difficulty here is in the already discussed extremely slow energy transport through the chain. In the case of oscillators the situation is much better but the problem is not  yet closed at all \cite{EPRB, EH,  RBT, C, HM, CEHRB}.

One of the most effective approaches to this problem consist in constructing a strict Lyapunov function for high energy regime, that controls the energy decay rate in the chain provided by dissipation at its ends; our paper is devoted to studying of this question. \footnote{Although we put dissipation only at one end of the chain, our results can be straightforwardly extended to the case when the dissipation acts at the both ends.} 
Unfortunately, the construction we present  gives a Lyapunov function only for the deterministic system but not for its white noise stochastic perturbation, used in the discussed above questions of nonequilibrium statistical mechanics.
It does not allow to prove mixing even for the chain of $3$ rotators coupled to thermal baths, in contrast to  the approach suggested in \cite{CEP_3, CE_4, CEW}.
However, our construction is simple, universal, requires only minimal assumptions and allows to get the rate of energy dissipation which seems to be optimal.
We believe that it is interesting by itself  and hope that its suitable modification could allow to make a progress in understanding of the mixing problem.

\subsection{Organization of the text} In Section~\ref{s:LF} we recall some basic facts from the Lyapunov function theory. In Proposition~\ref{p:strict_Lyap}  we present an adapted for our purpose construction from \cite{MM}, allowing to get a strict Lyapunov function (outside a compact set) once a non-strict one is given, under mild assumptions. This general construction applied to the system of rotators does not lead to the desired energy dissipation rate, but clarifies the result  from Section~\ref{s:rot}, obtained by its significant modification. There, in  Theorem~\ref{t:rotators} which is our main result, we find  a strict Lyapunov function for the system of rotators leading to the claimed estimate of the energy dissipation rate.  
Section~\ref{s:rot} can be read independently, but Section~\ref{s:LF} could help to understand the presented there construction. 
In Section~\ref{s:osc} we introduce the system of oscillators and check that is satisfies Assumption~\ref{asum1} from Section~\ref{s:LF}, crucial for our construction and allowing to apply Proposition~\ref{p:strict_Lyap} to get a strict Lyapunov function. In Section~\ref{s:mixing} we discuss in more details the problem of mixing in chains we concerned with in Section~\ref{s:mot}.

\bsk

\noindent {\bf Acknowledgments.} We are grateful to No\'e Cuneo and Sergei B. Kuksin for discussions.

\section{Lasalle's systems and strict Lyapunov functions}
\lbl{s:LF}

\subsection{Barbashin-Krasovskii Theorem, Lasalle's Invariance Principle}

We start by recalling a number of results on Lyapunov's direct method
of establishing asymptotic stability of equilibrium for a nonlinear system
via existence of Lyapunov function.

Consider a smooth $d$-dimensional manifold
\footnote{  In applications below $\cM$ will be chosen as a phase space of considered system, i.e. $\cM = \R^{2N}$ for the chain of $N$ oscillators and $\cM = \R^N\times \T^N$ in the case of rotators.}
 $\cM$ and a $C^{m}$-smooth function $F:\cM\mapsto\R$, $m\ge 1$.
Let $\hat x\in\cM$  be an equilibrium point for the system
\begin{equation}\label{eq_ODE}
	\dot x=F(x), \quad x\in \cM.
\end{equation}
Below we denote by $L_F W$ the Lie derivative of a function $W$ along the vector field $F$,
$$
L_F W:= \operatorname{grad} (W) \cdot F.
$$   
We recall that a function $W:\cM\mapsto\R$ is called {\it proper} if  sets
$\{x: \, W(x) \le c\}$ are compact for any $c\in\R$. 
\begin{defn}\lbl{d:LF}
	A  $C^1$-smooth function $W: \cM \to \mathbb{R}$ 
 is called {\it
	 strict
	Lyapunov function for the equilibrium $\hat x$} of \eqref{eq_ODE}, if i) $W$ is proper;  ii) $W(x)$ is
	positive  in $\cM \setminus \hat x$; iii) $W(\hat x)=0$, and
	\begin{equation}\label{dot_ineq}
		\mbox{iv}) \    L_FW(x)  < 0, \quad \forall x  \in \cM \setminus \hat x. 
	\end{equation}
\end{defn}

Classical A.M. Lyapunov's theorem claims   that if a strict Lyapunov function exists, then
$\hat x$ is asymptotically stable.

Assume now that $W(x)$ is a {\em non-strict}  Lyapunov function, i.e. it satisfies the non-strict version of   assumption \eqref{dot_ineq}
\begin{equation}\label{dot_leq}
	L_F W(x)  \leq 0, \quad \forall x  \in \cM .
\end{equation}
Asymptotic stability does not necessarily  hold in  this case,
unless  the set in which $L_FW(x)=0$ satisfies additional assumptions.
Most well-known criteria of asymptotic stability in this case are the two related results:
Barbashin-Krasovskii theorem
and  Lasalle's Invariance Principle.
\begin{prop}[Barbashin-Krasovskii-LaSalle theorem]\label{LaSalle}
	Let $\hat x$ be an equilibrium of system \eqref{eq_ODE} and
	$W: \cM \to \mathbb{R}$  be a non-strict Lyapunov function, which satisfies \eqref{dot_leq}.
	Let $\mathcal{P} \subset \cM$
	be a closed bounded neighbourhood of $\hat x$ which is positively invariant for \eqref{eq_ODE}.
	If there is no entire trajectory of the system, which is contained in
	$ \mathcal{P}  \setminus \hat x$ and on which $W$ is constant (the same, inequality  \eqref{dot_leq} turns equality),
	then $\hat x$ is asymptotically stable and $\mathcal{P} $ is contained in the basin of attraction
	of $\hat x$ for the system.
\end{prop}
See e.g. \cite[Section 4.2]{Kh}.
Note that if an entire trajectory $x(t)$ of  system \eqref{eq_ODE} is contained in
$ \mathcal{P}  \setminus \hat x$  and $W(x(t))$ is constant, then  there holds
\begin{equation}\label{lfkw}
	L_F^k W(x(t))=0, \ k=1, 2  \ldots
\end{equation}
at each point of the trajectory $x(t)$.
The following assumption would preclude the fulfilment of \eqref{lfkw}. 
Recall that the vector field $F$ is assumed to be $C^m$-smooth.

\begin{asmp}\label{asum0} There is  $1\le r \le m+1$ such that
the non-strict Lyapunov function $W$ is $C^r$-smooth  and
	\begin{equation}\label{eq_Lie_rk}
		\left(L_FW(x), L^2_FW(x), \ldots , L^r_FW(x) \right) \neq 0, \qquad  \forall x \in \cP\sm\hat x.
	\end{equation}
\end{asmp}

One  gets  an obvious corollary of Proposition \ref{LaSalle}.
\begin{cor}\label{Lie_rank} Let $W: \cM \to \mathbb{R}$  be a non-strict Lyapunov function for the equilibrium $\hat x$ of system \eqref{eq_ODE}.  
	Let $\cP\in\cM$ be a closed bounded neighbourhood of $\hat x$ which is positively invariant for the
	system.  Then, under Assumption~\ref{asum0}, the equilibrium $\hat x$ is asymptotically stable and $\cP$
	is contained in the basin of attraction of $\hat x$ for the system.
\end{cor}

\subsection{Strict Lyapunov Function for Lasalle's System}

Converse Lyapunov theorem claims that for each asymptotically  stable equilibrium
one can find a strict Lyapunov function \cite{Mas}.
In this section we discuss a "constructive version"\ of this result appropriate for our needs, which holds under a version of Assumption~\ref{asum0}. 
To this end we first extend  Definition \ref{d:LF} as follows.

\begin{defn}\lbl{LF_OC}
	{\it i)} A $C^1$-smooth function $W(x)$ is called {\it
		strict Lyapunov function for \eqref{eq_ODE} outside a compact set}, if  
		it is  proper and
		there is $Q\in \R$ such that 
	\begin{equation}\lbl{SLFoC}
	  L_FW(x)  < 0, \qquad \forall x  \in \cM \qmb{satisfying}\qu W(x) > Q.
	\end{equation} 
{\it ii)} It is called a {\it non-strict} Lyapunov function outside a compact set, if it is proper and for some $Q\in\R$ satisfies the non-strict version inequality \eqref{SLFoC}
\be\non
  L_FW(x)  \le 0, \qquad \forall x  \in \cM \qmb{satisfying}\qu W(x) > Q.
\ee
\end{defn}
Given a proper function $f$, we define  compact sets 
$$
\cK_Q^f := \{x\in\cM:\, f(x) \le Q\}, \qquad Q\in\R.
$$	
If in Definition~\ref{LF_OC} we wish to specify $Q$, we will write that $W$ is a strict or non-strict Lyapunov function outside the compact set $\cK_Q^W$.

For example, 

$\bullet$ a strict Lyapunov function for the equilibrium $\hat x$ of system \eqref{eq_ODE} is a  strict Lyapunov function  outside any compact set $\cK_Q^W$ with $Q\ge 0$;

$\bullet$ any proper function $W$ satisfying $L_F W(x) \le 0$ for all $x\in\cM$ is a  non-strict Lyapunov function  outside any compact set $\cK_Q^W$, $Q\in\R$.

\ssk

We have the following immediate result. 
\begin{lem}
	If $W$ is a strict Lyapunov function outside a compact set $\cK_Q^W$ then the  distance
	$$
	\dist(x(t), \cK_Q^W) \to 0 \qmb{as}\qu t\to\infty
	$$
	for any solution $x(t)$ to \eqref{eq_ODE}.
\end{lem}
{\it Proof.}
Since $L_F W<0$ outside $\cK_Q^W$, the latter set is  positively invariant. Thus, if $x(t)\in \cK_Q^W$ for some $t$ then we are done.
Otherwise, $W(x(t))$ monotonically decays to some $W_\infty\ge Q$, so that 
\be\lbl{x(t) comp}
Q\le W_\infty \le W(x(t)) \le W(x(0))  \qmb{for all $t\ge 0$}.
\ee 
Suppose that $W_\infty \ne Q$. Then $x(t)$ for all $t\ge 0$  belongs to a compact set which does not intersect with $\cK_Q^W$. Then, $\inf_{t\ge 0} L_F W (x(t))<0$ which contradicts \eqref{x(t) comp}.
\qed

\ssk

 Now let $Q\in\R$ and $W: \cM \to \mathbb{R}$ be a non-strict Lyapunov function outside a compact set $\cK_Q^W$ for system \eqref{eq_ODE}. 
 We will need the following version of Assumption~\ref{asum0}. 
	\begin{asmp}\label{asum1}  There exists  $1\le r \le m$ such that the Lyapunov function $W$ is $C^{r+1}$-smooth and satisfies relation \eqref{eq_Lie_rk} of Assumption~\ref{asum0} outside the compact set~$\cK_Q^W$,
		\begin{equation}\non
			\left(L_FW(x), L^2_FW(x), \ldots , L^r_FW(x) \right) \neq 0, \qquad  \forall x \notin \cK_Q^W.
		\end{equation}
	\end{asmp}

Under this assumption there exist smooth positive functions $\phi,\,\Phi:  (Q,\infty)\mapsto \R_{>0}$, for any $x\notin \cK_Q^W$ satisfying 
\begin{align}
	\lbl{asum_phi}	
	&  \left|L_FW(x)\right|+\sum_{k=2}^{r}\big|L^k_FW(x)\big|^2 \geq \varphi(W(x)), 
	\\\lbl{asum_Phi}
	&  |L^k_FW(x)| \leq  \Phi(W(x)) \quad \mbox{for all}\quad 1 \leq k \leq r+1,
\end{align}
and
\be\lbl{>1}
\phi\le\Phi^2.
\ee
Indeed, one can take 

\[\wt\varphi(w)=\inf\left\{\left.\big|L_FW(x)\big|+\sum_{k=2}^r\big|L_F^kW(x)\big|^2\ \quad \right| x\in\cM:\, W(x) = w \right\}\]
and
\[\left. \wt \Phi(w)=\sup\left\{\big|L^{k}_FW(x)\big| \quad \right|k=1, \ldots , r+1; \ x\in\cM:\, W(x) = w  \right\},   \]
and then minorize and majorize functions $\wt\phi$ and $\wt\Phi$ by smooth positive functions $\phi$ and $\Phi$ correspondingly, satisfying \eqref{>1}.

Below we fix a pair of functions $\phi$ and $\Phi$ satisfying \eqref{asum_phi}, \eqref{asum_Phi} and \eqref{>1}. 
\begin{prop}\label{p:strict_Lyap} Let $Q\in\R$ and $W: \cM \to \mathbb{R}$ be a non-strict Lyapunov function outside a compact set $\cK_Q^W$
	for system \eqref{eq_ODE}. Then, under
Assumption~\ref{asum1},
there exist smooth functions $A, B_k:(Q,\infty)\mapsto\R$, $k=2,\dots,r$, such that $B_k\ge 1$ and for any $\eps>0$
the function
\begin{equation}\label{eq_sLyap}
	W^\sharp(x):=A(W(x))-\sum_{k=2}^r B_k(W(x))L^{k-1}_FW(x)L^k_FW(x), \quad x\notin \cK^W_{Q+\eps},
\end{equation}
continued anyhow to a $C^1$-smooth function on $\cM$,	is a strict  Lyapunov function for \eqref{eq_ODE}  outside $\cK^W_{Q+\eps}$ 	(the  $C^1$-smooth continuation always exists).

The functions $B_k$ can be chosen in the form
\be\lbl{B_k_form}
B_k =  2^{(r-k)(r-k+1)}\big(\Phi^2 / \varphi\big)^{r-k}, \qquad k=2,\dots, r,
\ee
while the function $A$ can be taken arbitrary, growing at infinity so fast that  $W^\sharp$ is proper,    with derivative satisfying
	\be\lbl{asum_A}
A' > \Phi^2 \sum_{k=2}^r |B'_k|  + \Phi B_2 + 1.
	\ee
With  this choice,
\[
L_F W^\sharp(x) \le -\phi(W(x))/4, \qquad \forall x\notin \cK_{Q+\eps}^W.
\]
\end{prop}
This is a version of result  established in \cite[Th. 5.1]{MM}. 
We provide a proof which is more succinct than the original one.

\subsection{Proof of Proposition \ref{p:strict_Lyap}}

	 Differentiating  function \eqref{eq_sLyap} along the direction of the vector field $F$ we get
	\begin{eqnarray}
		L_FW^\sharp (x) = \label{LFW} \\
		\left[A'(W(x))-\sum_{k=2}^rB'_k(W(x))L_F^{k-1}W(x)L_F^{k}W(x)\right]L_FW(x)- \nonumber \\
		\sum_{k=2}^rB_k(W(x))\left(L_F^kW(x)\right)^2 -  \nonumber \\
		B_2(W(x))L_FW(x)L_F^{3}W(x) - \sum_{k=3}^{r}B_k(W(x))L_F^{k-1}W(x)L_F^{k+1}W(x)=  \nonumber \\
		A_1(x)  L_FW(x) -  \sum_{k=2}^rB_k(W(x))\left(L_F^kW(x)\right)^2 - \label{LFsharp2} \\
		\sum_{k=3}^{r}B_k(W(x))L_F^{k-1}W(x)L_F^{k+1}W(x),  \label{LFsharp3}
	\end{eqnarray}
	where
	\begin{equation}\non
		A_1(x)=
		\left[A'(W(x))-\sum_{k=2}^rB'_k(W(x))L_F^{k-1}W(x)L_F^{k}W(x)-  B_2(W(x))L_F^{3}W(x)\right]
	\end{equation}
	Since we suppose $B_k\ge 1$ all $k$,  the  sum in \eqref{LFsharp2}
	is non-negative.
	By \eqref{asum_Phi},
	$$
	A_1 \ge A'(W) - \sum_{k=2}^r |B'_k(W)| \Phi(W)^2 - B_2(W)\Phi(W).
	$$
	Choosing the function $A$ satisfying  \eqref{asum_A}, we get $A_1 \ge 1$. In particular, the term \eqref{LFsharp2} is non-positive, since $L_FW\leq 0$.

	Next, we choose $B_k$ in such a way that \eqref{LFsharp3} is bounded by \eqref{LFsharp2}. More specifically,  we assume that 
	\begin{equation}\label{Bk-mean}
		B_k\le \frac{(B_{k-1})^{1/2}}{\sqrt2}\frac{(B_{k+1})^{1/2}}{\sqrt2}, \quad 3\le k \le r-1.
	\end{equation}
	Then 
	$$
	\left| B_k(W)L_F^{k-1}  W  L_F^{k+1}W \right|  \leq \frac{B_{k-1}(W)}{4}\left(L_F^{k-1}W\right)^2+\frac{B_{k+1}(W)}{4}\left(L_F^{k+1}W\right)^2,
	$$
	so that 
	\begin{align*}
		\sum_{k=3}^{r-1}   \left|B_k(W)L_F^{k-1}  W  L_F^{k+1}W\right| &\le \sum_{k=4}^{r-2} \frac{B_{k}(W)}{2}\left(L_F^{k}W\right)^2
		+  \sum_{k=2,3,r-1,r} \frac{B_{k}(W)}{4}\left(L_F^{k}W\right)^2
		\\ &\le \sum_{k=2}^{r} \frac{B_{k}(W)}{2}\left(L_F^{k}W\right)^2
		-  \frac{B_{r-1}(W)}{4}\left(L_F^{r-1}W\right)^2.
	\end{align*}
	Similarly, 
	$$
	|B_r(W)L_F^{r-1}W  L_F^{r+1}W| \le \frac{B_{r-1}(W)}{4}\left(L_F^{r-1}W\right)^2 +  \frac{(B_r(W))^2}{B_{r-1}(W)}\left(L_F^{r+1}W\right)^2. 
	$$
	Inserting these estimate into  \eqref{LFsharp2}-\eqref{LFsharp3},  we find
	\begin{align*}
		L_FW^\sharp \le A_1L_FW -  \frac12\sum_{k=2}^{r} B_k(W)\left(L_F^kW\right)^2 +\frac{(B_r(W))^2}{B_{r-1}(W)}\left(L_F^{r+1}W\right)^2.
	\end{align*}
	Now we set $B_{r}:=1$. Then, according to~\eqref{asum_phi}
	$$
	A_1L_FW -  \frac12\sum_{k=2}^{r} B_k(W)\left(L_F^kW\right)^2\le -\frac{\varphi(W)}{2},
	$$
	since  $A_1, B_k \ge 1$.
	On the other hand, 
	$$
	\frac{(B_r(W))^2}{B_{r-1}(W)}\left(L_F^{r+1}W\right)^2 \le \frac{(\Phi(W))^2}{B_{r-1}(W)}.
	$$
	We  choose $B_{r-1}$ in such a way that $\displaystyle{ \frac{(\Phi(W))^2}{B_{r-1}(W)}  = \frac{\varphi(W)}{4}}$,
	i.e.
	\begin{equation}\label{Br-1}
		B_{r-1} = \frac{4\Phi^2}{\varphi}.
	\end{equation}
	Note that $B_{r-1}\ge 1$, due to \eqref{>1}. Then, 
	$$
	L_F W^\sharp \le - \frac{\varphi(W)}{4}.
	$$
	It remains to choose the functions $B_k\ge 1$, $k<r-1$, satisfying \eqref{Bk-mean}, which is equivalent to
	$$
	B_{k-1} \ge \frac{4 B_k^2}{B_{k+1}}.
	$$
	We choose them to be minimal, that is replace the inequality above by the equality. Since $B_{r}=1$ and $B_{r-1}$ is given by \eqref{Br-1}, we find 
	\be\lbl{B_r_def}
  B_k =  2^{(r-k)(r-k+1)}\Big(\frac{\Phi^2}{\varphi}\Big)^{r-k}.
	\ee

 Finally, we note that the constructed in such a way function $W^\sharp$ is defined outside the compact set $\cK^W_{Q}$ and is $C^1$-smooth. Indeed,   $W(x)>Q$ for $x\notin \cK^W_{Q}$, so the functions $\phi(W(x))$ and $\Phi(W(x))$ are well-defined and the function $\phi(W(x))$ appearing in the denominator of \eqref{B_r_def} does not vanish.

To get from  $W^\sharp$ the desired strict Lyapunov function, which we denote $W_\eps^\sharp(x)$, we restrict $W^\sharp$ to the set $\cK_{Q+\eps}^W$ and $C^1$-smoothly continue it to the whole manifold $\cM$.
Such continuation exists since we can take
\[
 W_\eps^\sharp(x):= 
\left\{
\begin{array}{cl}
W^\sharp(x)\chi_\eps(W(x)) &\qmb{if } x\notin \cK_{Q}^W, \\
0 &\qmb{if } x\in \cK^W_Q,
\end{array}
\right.
\]
where $\chi$ is a smooth function satisfying $\chi(w)=1$ for $w\ge Q+\eps$ and $\chi(w)=0$ for $w\le Q+\eps/2$. 
\qed

%

\section{Chain of rotators}
\lbl{s:rot}

\subsection{Lyapunov function  and energy decay rate}
In this section we study the chain of rotators given by \eqref{H_rot}-\eqref{eq_intro}. 
 We  construct a strict Lyapunov function  outside a compact set that provides  lower bound \eqref{en_dec_intro} for the energy dissipation rate.

Writing the equations of motion \eqref{eq_intro} in more details, we get
\be\lbl{eq}
\dot q_j = p_j, \qquad \dot p_j = -\de_{j1} p_1 + V_{j-1}'(q_{j-1}-q_j) - V_j'(q_{j}-q_{j+1}), \quad 1\le j \le N,
\ee
where $V_0=V_N:= 0$.
Clearly, the Hamiltonian $H$ is a non-strict Lyapunov function for system \eqref{eq}  outside a compact set $\cK_Q^H$ with arbitrary $Q\in\R$, since it is proper and
\be\lbl{dot_H}
L_F H = - p_1^2 \le 0.
\ee
Here $F$ denotes the vector field from the r.h.s. of \eqref{eq} and $L_F$ is the Lie derivative along its direction.  
We assume that the interaction potentials $V_j$ satisfy
\begin{asmp}\lbl{asumR}
	Potentials $V_i:\,\T\mapsto \R$ are $C^3$-smooth 
	and $(V_i'(x))^2 + (V_i''(x))^2\ne 0$ for every $x\in\T$ and $1\le i \le N-1$.
	Moreover, without loss of generality we assume that $V_i\ge 1$ for any $i$, so $H\ge 1$ as well.
\end{asmp}

By simplifying proof of Proposition~\ref{thm: oscillators simple} below, one can easily check that Assumption~\ref{asumR} implies Assumption~\ref{asum1} with appropriate $Q$ and $r=4N-3$, once the potentials $V_j$ are $C^{r+1}$-smooth. So, Proposition~\ref{p:strict_Lyap} with $W=H$ provides a strict Lyapunov function $W^\sharp$ outside a compact set for system \eqref{eq}. However, it is difficult to identify the function $\phi$ from \eqref{asum_phi}, and consequently $W^\sharp$. 
 Moreover, the function $\Phi$ in \eqref{asum_Phi} grows at least  as $C p_1^{r}$ when $|p_1|\to\infty$ since $L_F^{r+1}H$ does, so the function $B_2$ from \eqref{B_k_form} grows at least as $C_1 p_1^{2r(r-2)}/\phi^{r-2}$. Our estimations for the function $\phi$ show that this gives too large upper bound for the function $W^\sharp$, leading to a bad lower bound for the energy decay rate. 

To sharpen the bound, below we suggest another  strict Lyapunov function outside a compact set, construction of which is motivated by \eqref{eq_sLyap}, but instead of the derivatives $L^k_F H$  we employ their appropriate parts. This construction is also related to the one presented in \cite[Theorem 3.1]{MM}.

\smallskip

Let us denote
\be\lbl{xi}
\xi_j(q):=- V_j'(q_{j}-q_{j+1}), \qmb{so that}\qu L_F p_j = \xi_j - \xi_{j-1} - \de_{1j}p_1,
\ee 
for $1 \le j \le N-1$, where $\xi_0:=0$.
Consider the function 
\be\lbl{LF_rot}
W(p,q) = a_0 H^{\ga_0} - \sum_{j=1}^{N-1}
\big(a_{2j-1} H^{\al_{2j-1}}p_j\xi_j  + a_{2j} H^{\al_{2j}}\xi_j L_F\xi_j\big),
\ee
where the constants $a_k \ge 1$ and 
\be\lbl{albet}
\ga_0:= 2N-1, \qquad \al_k := 2(N-1)-k \qmb{for}\qu 1\le k\le 2(N-1).
\ee
Due to the bounds
\be\lbl{pxiLxi}
|p_j\xi_j|\le C\sqrt H, \qquad |\xi_jL_F\xi_j|\le C\sqrt H,
\ee
the function $W$ satisfies
\be\lbl{HWH}
\frac{a_0}{2} H^{\ga_0} \le a_0H^{\ga_0} \big(1-CH^{-3/2}\big) \le W \le a_0H^{\ga_0}\big(1+CH^{-3/2}\big) \le 2a_0 H^{\ga_0},
\ee
since we assume $a_0$ to be is sufficiently large.
Here and below by $C,C_1,\dots$ we denote various positive constants which may depend on $N$ and change from line to line. 

\begin{thm}\lbl{t:rotators}
	Let $N\ge 2$,  the constants $a_{2N-2}$ and $a_k-a_{k+1}$  are sufficiently large for all $k$, and
	 Assumption~\ref{asumR} is satisfied. Then 
	 $W$ is a strict Lyapunov function outside a compact set  and 
	\be\lbl{L_form}
	L_F W \le -H + C_1 \le -C_2W^{1/\ga_0} + C_1,	
	\ee
	for appropriate constants $C_{1,2}=C_{1,2}(N; a_0,\dots, a_{2N-2})$.
\end{thm}

In the following corollary we analyse the energy decay rate in the system of rotators. 
Let $(p,q)(t)$ be a solution to \eqref{eq}. Below we denote $W_t:=W(p,q)(t)$ and $H_t:=H(p,q)(t)$. Recall that $\ga_0 = 2N-1$.
\begin{cor}\lbl{c:en_dec}
Under Assumption~\ref{asumR},	there are $N$-dependent constants $h_0\ge 1$, $0<\eps<1$ and $C>0$ such that for $H_0\ge h_0$, 
\be\lbl{W_tW_0}
W_t\le W_0 - Ct W_0^{1/\ga_0} \qmb{for any}\qu 0\le t\le \eps H_0^{\ga_0-1}
\ee 
and 
\be\lbl{H_tH_0}
H_t\le H_0 - \frac{Ct}{H_0^{\ga_0-2}} \qmb{for any}\qu \eps^{-1} H_0^{\ga_0-\frac52}\le t\le \eps H_0^{\ga_0-1}.
\ee 
\end{cor}
We call $CH^{-(\ga_0-2)}$ {\it energy decay rate} of system \eqref{eq} on the time interval from \eqref{H_tH_0}.
 
{\it Proof.} Let 
\[t_{stop} = \max\{0\le t\le \eps H_0^{\ga_0-1}: \; W_t \ge C^{-1}_{stop}W_0\}\]
for sufficiently large constant $C_{stop}$.
Since by \eqref{HWH} $W_0\ge Ch_0^{\ga_0}$, for  sufficiently large $h_0$ by \eqref{L_form} we have
\[
\dot W_t \le -C W_t^{1/\ga_0} \qmb{for}\qu 0\le t\le t_{stop}.
\]
 Integrating this inequality,  we get
\[
W_t^{(\ga_0-1)/\ga_0} \le W_0^{(\ga_0-1)/\ga_0} -Ct, \qquad t\le t_{stop}.
\]
Accordingly, 
\be\lbl{Wt-W0}
W_t \le W_0\Big(1-\frac{Ct}{W_0^{(\ga_0-1)/\ga_0}}\Big)^{\ga_0/(\ga_0-1)}
\le W_0\Big(1-\frac{C_1t}{W_0^{(\ga_0-1)/\ga_0}}\Big),
\ee
 since the ratio $\ds{\frac{C_1t}{W_0^{(\ga_0-1)/\ga_0}}}$ is small for $t\le t_{stop}\le \eps H_0^{\ga_0-1}$ once $\eps$ is. Thus, we get \eqref{W_tW_0} for $0\le t\le t_{stop}$. For $t_{stop}\le t\le \eps H_0^{\ga_0-1}$ we have $t\le  \eps C W_0^{(\ga_0-1)/\ga_0}$, so  \eqref{W_tW_0} also holds once $\eps$ is sufficiently small, since $W_t\le C^{-1}_{stop}W_0$ for $t\ge t_{stop}$.

Due to \eqref{HWH}, $H_t \le H_0/2$ for $t\ge t_{stop}$ once the constant $C_{stop}$ is sufficiently large. Then, \eqref{H_tH_0} holds true for $t_{stop}\le t \le \eps H_0^{\ga_0-1}$.
By \eqref{HWH} and \eqref{Wt-W0}, for $0\le t\le t_{stop}$,
\begin{align}
	H_t &\le H_0\left(\frac{1+C H_0^{-3/2}}{1-CH_t^{-3/2}}\right)^{1/\ga_0} \left(1 - \frac{Ct}{H_0^{\ga_0-1}}\right)^{1/\ga_0}
	\\ \non
	&\le \left(1 + C_1H_t^{-3/2}\right)\left(H_0 - \frac{C_2t}{H_0^{\ga_0-2}}\right) 
	\le H_0\big(1+ C_3 H_0^{-3/2}\big) - \frac{C_2t}{H_0^{\ga_0-2}},
\end{align}
where we have used that $C^{-1} H_0 \le H_t\le H_0$ in the considered time interval.  Then, for $t\ge \eps^{-1}H_0^{\ga_0-\frac52}$ we conclude
\eqref{H_tH_0}.
\qed

\subsection{Proof of the theorem}
We abbreviate $V_j:=V_j(q_j-q_{j+1})$. Unless otherwise is explicitly stated, in this proof constants $C, C_1, C_2, \dots$  never depend on the parameters $a_i$.
Without loss of generality we assume that the energy $H$ satisfies $H\ge h_0$ for sufficiently large $h_0$. Indeed, for $(p,q)$ such that $H(p,q)<h_0$ the r.h.s. of equation  \eqref{dot_L} below is bounded by a constant $C=C(h_0,a_0,\dots, a_{2N-2})$. 

\noindent {\it Step 1.} By \eqref{dot_H},
\begin{align}\non
	L_F &W =  \\\non
	& - \left(a_0\ga_0 H^{\ga_0-1} 
	- \sum_{j=1}^{N-1} \Big( a_{2j-1}\al_{2j-1} H^{\al_{2j-1}-1}p_j\xi_j + a_{2j}\al_{2j} H^{\al_{2j}-1}\xi_j L_f\xi_j\Big)\right)p_1^2 
	\\\lbl{dot_L}
	&- \sum_{j=1}^{N-1}\Big(a_{2j-1}H^{\al_{2j-1}} L_F(p_j\xi_j) + a_{2j}H^{\al_{2j}}L_F(\xi_jL_F\xi_j)\Big).
\end{align}
We start with estimating the terms of the last sum in \eqref{dot_L}. By \eqref{xi},  
$$
L_F(p_j\xi_j) =\xi_j L_Fp_j +  p_jL_F\xi_j = \xi_j(\xi_j - \xi_{j-1} - \de_{j1} p_1) + p_jL_F\xi_j.
$$
Let us bound $|p_jL_F\xi_j|\le \Ga_j H p_j^2 + \Ga_j^{-1}H^{-1}\big(L_F\xi_j\big)^2$, where constants $\Ga_j$ will be chosen later to be sufficiently large. Then, applying twice the inequality 
$|xy|\le x^2/4 + y^2$ to $xy=\xi_j\xi_{j-1}$ and $xy=\xi_j(\de_{j1}p_1)$,  
we find
\be\lbl{dot_pxi}
 L_F(p_j\xi_j) \ge \frac{\xi_j^2}{2} - \xi_{j-1}^2 - \de_{j1}p_1^2 
- \Ga_j H p_j^2  - \Ga_j^{-1} H^{-1} \big(L_F\xi_j\big)^2.
\ee
Next, we note that
\be\lbl{Lxi}
L_F \xi_j = - (p_j - p_{j+1})V_j''.
\ee
Differentiation $L_F$ applied to \eqref{Lxi} gives us
$$
L^2_F \xi_j = -(p_j-p_{j+1})^2V_j''' - (L_F p_j - L_F p_{j+1}) V_j''
$$
and by \eqref{xi},
$$
\big(L^2_F \xi_j\big)^2 \le C(p_j^4 + p_{j+1}^4 + \xi_{j-1}^2 + \xi_j^2 + \xi_{j+1}^2 + \de_{j1}p_1^2).
$$
Then, 
\begin{align}\non
L_F(\xi_jL_F\xi_j) &= (L_F\xi_j)^2 + \xi_j L^2_F\xi_j 
\ge (L_F \xi_j)^2 - a_{2j} H \xi_j^2 - (a_{2j} H)^{-1}(L_F^2 \xi_j)^2
\\\non &\ge   
(L_F\xi_j)^2 - a_{2j}H\xi_j^2 
\\ \lbl{LxiLxi}
&- 
 C(a_{2j}H)^{-1}(p_j^4 + p_{j+1}^4 + \xi_{j-1}^2 + \xi_j^2 + \xi_{j+1}^2 + \de_{j1}p_1^2).
\end{align}

\noindent {\it Step 2.} Substituting the r.h.s. of estimates \eqref{dot_pxi} and \eqref{LxiLxi} into the last sum in \eqref{dot_L},
we bring together positive terms of the result into
\be\lbl{I_pos}
I_{pos}:=\sum_{j=1}^{N-1}\Big(a_{2j-1}H^{\al_{2j-1}}\frac{\xi_j^2}{2} + a_{2j}H^{\al_{2j}}(L_F\xi_j)^2\Big).
\ee
 Applying the inequality
$(x-y)^2 \ge x^2/2 - y^2$
to the squared \eqref{Lxi}, we find 
\be\lbl{xi+p}
(L_F\xi_j)^2 \ge \frac{(V_j'')^2}{2}p_{j+1}^2 - Cp_j^2. 
\ee
Then, using that $\xi_j = -V_j'$ and $H\ge p_{j+1}^2/2$, we obtain by \eqref{xi+p}
\be\non
H\xi_j^2 + (L_F\xi_j)^2 \ge \frac{(V_j')^2 + (V_j'')^2}{2}p_{j+1}^2 - Cp_j^2 \ge C_1^{-1}p_{j+1}^2 -Cp_j^2,
\ee
according to Assumption~\ref{asumR}.
Recalling that $\al_{2j-1} = \al_{2j}+1$ and assuming $a_{2j-1}/2 \ge a_{2j}$, we find
\be\non
I_{pos}\ge \sum_{j=1}^{N-1}a_{2j}H^{\al_{2j}} \big(C_1^{-1}p_{j+1}^2 - Cp_j^2\big)
\ge \frac{C_1^{-1}}{2}\sum_{j=2}^N a_{2j-2}H^{\al_{2j-2}}p_j^2 - Ca_2H^{\al_2}p_1^2.
\ee
Finally, we bound a half of the term $I_{pos}$ by the estimate above and keep another half as it is, 
\begin{align}\lbl{Iposest}
I_{pos}\ge\frac12\sum_{j=1}^{N-1}\Big(a_{2j-1}H^{\al_{2j-1}}\frac{\xi_j^2}{2} &+ a_{2j}H^{\al_{2j}}(L_F\xi_j)^2\Big) 
\\\non
&+ C^{-1}\sum_{j=2}^N a_{2j-2}H^{\al_{2j-2}}p_j^2 - Ca_2H^{\al_2}p_1^2.
\end{align}

\noindent {\it Step 3.} Now we 
estimate the r.h.s. of \eqref{dot_L} using bounds \eqref{pxiLxi}, \eqref{dot_pxi}, \eqref{LxiLxi} and \eqref{Iposest}:
\be\non
L_F W \le -(I_{p_1} + I_{\xi} + I_{\dot\xi} + I_p),
\ee
where
\begin{align}\non
I_{p_1} &= p_1^2\,\Big(a_0\ga_0 H^{\ga_0-1} - C\sum_{j=1}^{N-1} \big(a_{2j-1}\al_{2j-1} H^{\al_{2j-1}-1/2} + a_{2j}\al_{2j} H^{\al_{2j}-1/2}\big) 
\\\non
& - a_1H^{\al_1} - \Ga_1a_1H^{\al_1+1} 
- CH^{\al_2-1}(p_1^2 + 1) -Ca_2H^{\al_2}\Big),
\end{align}
the term $I_{\xi}$ is given by
\ben
I_{\xi} = \sum_{j=1}^{N-1}\xi_j^2\,\Big(\frac{a_{2j-1}}{4}H^{\al_{2j-1}} - a_{2j+1}H^{\al_{2j+1}} -a_{2j}^2H^{\al_{2j}+1}
-C\big(H^{\al_{2j-2}-1} + H^{\al_{2j}-1} + H^{\al_{2j+2}-1}\big) \Big);
\een
here and below we denote $a_{2N-1}=a_{2N}=\Ga_N:=0$ and $\al_{2N-1}=\al_{2N}=\al_0:=-\infty$. The term $I_{\dot\xi}$ has the form
$$
I_{\dot\xi} = \sum_{j=1}^{N-1} (L_F\xi_j)^2\,\Big(\frac{a_{2j}}{2}H^{\al_{2j}} - \Ga_j^{-1}a_{2j-1}H^{\al_{2j-1}-1}\Big)
$$
and finally 
$$
I_p = \sum_{j=2}^{N} p_j^2\,\Big(C^{-1}a_{2j-2}H^{\al_{2j-2}} - \Ga_ja_{2j-1}H^{\al_{2j-1}+1} - Cp_j^2\big(H^{\al_{2j-2}-1} + H^{\al_{2j}-1}\big)\Big).
$$

Let $\Ga_j = 2a_{2j-1}/a_{2j}$. Then $I_{\dot\xi}=0$, since $\al_{2j-1}-1 = \al_{2j}$.
Next, using that $p_j^2\le 2H$, we find
\begin{align*}
I_p &\ge  \sum_{j=2}^{N} p_j^2 H^{\al_{2j-2}}\Big(C^{-1}a_{2j-2} - \Ga_ja_{2j-1} - C\big(1 +  H^{-2}\big)\Big)\,
\\
&\ge \sum_{j=2}^{N} p_j^2 H^{\al_{2j-2}} \ge H-C-\frac{p_1^2}{2},
\end{align*}
once $a_{2j-1}$ is sufficiently large in terms of $\Ga_ja_{2j-1} = 2a_{2j-1}^2/a_{2j}$.
Similarly,
$$
I_\xi\ge  \sum_{j=1}^{N-1}\xi_j^2 H^{\al_{2j-1}}\,\Big(\frac{a_{2j-1}}{4} - a_{2j+1}H^{-2} - a_{2j}^2 - C\big(1 + H^{-2} + H^{-4}\big)\Big) \ge 0,
$$ 
once $a_{2j-1}$ is sufficiently large in terms of $a_{2j+1}$ and $a_{2j}$.
Finally, 
$$
I_{p_1} \ge H^{\ga_0-1}\frac{p_1^2}{2} 
$$
once $a_0$ is sufficiently large.

The coefficients $a_k$ satisfying the restrictions above can be chosen iteratively.
Then, collecting together the estimates above, we get $I_{p_1} + I_{\xi} + I_{\dot\xi} + I_p \ge H-C$, so 
$$
L_F W \le -H + C.
$$
The last inequality in \eqref{L_form} holds due to \eqref{HWH}.
\qed

\section{Chain of oscillators}
\lbl{s:osc}

\subsection{The setting}
 In this section we establish the existence of a strict Lyapunov function outside a compact set for the chain of oscillators in presence of friction in the first oscillator.  To this end we show that Proposition~\ref{p:strict_Lyap} applies to the system under natural assumptions. We do not construct a Lyapunov function satisfying good bounds as we did for the chain of rotators in Section~\ref{s:rot}, since in the present case it can be done similarly.

 The Hamiltonian of the chain of oscillators is 
\be\lbl{H_osc}
H(p,q)=\sum_{j=1}^N \frac{p_j^2}{2} + \sum_{j=1}^N U_j(q_j) + \sum_{j=1}^{N-1} V_j(q_j-q_{j+1}),
\ee
where $N\ge 2$, $p=(p_1,\dots,p_N)\in\R^N$ denote moments of the oscillators while $q=(q_1,\dots,q_N)\in\R^N$ stand for their coordinates.
For simplicity, we assume the interaction potentials $V_k$ as well as the so-called pinning potentials $U_k$ to be smooth.
The equations of motion are given by \eqref{eq_intro}, or, in more details,
\begin{equation}
	\label{eq: main oscillators system with friction}
\dot q_j = p_j, \qquad \dot p_j = -\de_{j1} p_1 - U_j'(q_j) + V_{j-1}'(q_{j-1}-q_j) - V_j'(q_{j}-q_{j+1}), \quad 1\le j \le N,
\end{equation}
where $V_0=V_N:= 0$. If the Hamiltonian $H$ is proper, it gives a non-strict Lyapunov function outside a compact set $\cK^H_Q$ with any $Q\in\R$ since
\[
L_F H =  -p_1^2,
\]
where $F$ stands for the vector field from the r.h.s. of equation \eqref{eq: main oscillators system with friction}.

\subsection{Strictly convex case}
We start our analysis with a simple case of strictly convex potentials.
\begin{prop}
	\label{thm: oscillators simple}
	
	Suppose that for every $k$ the potentials $U_k(x)$ and $V_k(x)$ are strictly convex and have  minimum at $x=0$. Then for any $(p,q)\in\mathbb{R}^{2N}\sm\{0\}$ there exists $1\le p\le 4N-1$ such that $L_F^pH\ne 0$.
	
\end{prop}

\begin{cor}
	
 Under assumptions of Proposition~\ref{thm: oscillators simple}, system \eqref{eq: main oscillators system with friction} satisfies Assumption \ref{asum1} with $W=H$, $r=4N-1$ and $Q=H(0,0)$. Hence it admits a strict Lyapunov function $W^\sharp$ outside a compact set, constructed in Proposition~\ref{p:strict_Lyap}.
	
\end{cor}

\begin{proof}[Proof of Proposition~\ref{thm: oscillators simple}]

	Suppose that for some $(p,q)$ we have $L_F^kH=0$ for any $k\le 4N-1$. Since
	\[
	L_F^{k+1} H = -\sum_{j=0}^k \binom{k}{j} L_F^j p_1 L_F^{k-j}p_1,
	\]
	the following vanishings are equivalent:
	\[
	\forall 1\le k\le 4N-1:\ \ L_F^k H=0
	\quad\Leftrightarrow\quad
	\forall 0\le k \le 2N-1:\ \ L_F^k p_1=0.
	\]
	Denote $U_k=U_k(q_k)$ and $V_k=V_k(q_k-q_{k+1})$ for short. Then
	\[
	L_F p_1 = -U_1'-V_1'=0
	\quad\text{and}\quad
	L_F^2 p_1 = -(U_1''+V_1'')p_1 + V_1''p_2.
	\]
	Since $V_1''\ne 0$, if $p_1=L_Fp_1=0$ then
	\[
	\forall 2\le j\le 2N-1:\ L_F^j p_1 = 0
	\quad\Rightarrow\quad
	\forall 0\le j\le 2N-3:\ L_F^j p_2 = 0.
	\]
	Similarly,
	\[
	L_F p_2 = -U_2'+V_1'-V_2'
	\quad\text{and}\quad
	L_F^2 p_2 = V_1''p_1-(U_2''+V_2''+V_1'')p_2 + V_2''p_3.
	\]
	Again, since $V_2''\ne0$, if $p_1=L_Fp_1=p_2=L_Fp_2=0$, then
	\[
	\forall 2\le j\le 2N-3:\ L_F^j p_2 = 0
	\quad\Rightarrow\quad
	\forall 0\le j\le 2N-5:\ L_F^j p_3 = 0.
	\]
	Proceeding by induction, we obtain
	\[
	\forall 0\le j \le 2N-1:\ \ L_F^j p_1=0
	\quad\Rightarrow\quad
	\forall 1\le k\le N: \ \ p_k=L_Fp_k=0.
	\]
	So $p=0$ and equations $L_Fp_k=0$ form the following system:
	\begin{equation}
		\label{eq: main equilibrium system}
		\begin{cases}
			U_1'+V_1'=0\\
			U_2'-V_1'+V_2'=0\\
			\ldots\\
			U_{N-1}'-V_{N-2}'+V_{N-1}'=0\\
			U_N'-V_{N-1}'=0.\\           
		\end{cases}
	\end{equation}
Since $x=0$ is the unique minimum for the potentials $U_k$,  $U_k'(x)>0$ for $x>0$ and vice versa.
	Suppose that $q_1\ge 0$ (the case $q_1\le 0$ is similar). Then $U_1'\ge 0$ and the first equation implies
	\[
	V_1'\le 0\ \Rightarrow\ q_2\ge q_1\ge 0\ \Rightarrow\ 
	U_2'\ge0\ \Rightarrow U_2'-V_1' \ge0.
	\]
	So the second equation implies
	\[
	V_2'\le 0\ \Rightarrow\ q_3\ge q_2\ge 0\ \Rightarrow\ 
	U_3'\ge0\ \Rightarrow U_3'-V_2' \ge0,
	\]
	and so on. The penultimate equation implies
	\[
	V_{N-1}'\le0\ \Rightarrow\ q_N\ge q_{N-1}\ge 0\ \Rightarrow\ 
	U_N'\ge0\ \Rightarrow U_N'-V_{N-1}' \ge0,
	\]
	And the last equation is $U_N'-V_{N-1}'=0$. Hence $U_N'=0$ and $q_N=0$. Therefore $0=q_N\ge q_{N-1}\ge\ldots\ge q_1\ge 0$. So $q=0$.
\end{proof}

\subsection{General case}
 Now we extend the argument to a more general case when potentials $U_k$ and $V_k$ are strictly convex outside a compact set.

\begin{thm}
	\label{thm: oscillators complicated}
	
	Suppose that 
	 $V_k''(x)>0$ for $|x|\ge R$ with some $R\ge 0$, $U_k'(x)\to\pm\infty$
	as $x\to\pm\infty$, and $V_k''(x)\ne 0$ or $V_k'''(x)\ne 0$ at every $x\in\R$. Then there is a bounded set $K\subset\R^{2N}$ such that for any $(p,q)\in\mathbb{R}^{2N}\sm K$ there exists $1\le j\le 3\cdot 2^{N+1}-5$ for which $L_F^jH\ne 0$. 
\end{thm}
Note that under assumptions of the theorem, the Hamiltonian $H$ is proper.
 
 From the proof of the theorem it can be seen that 
\[K\subset \{p=0, \; -b<q_k<a \; \forall k\}\] 
with $a,b$ defined in \eqref{a_VU}, \eqref{b_VU}.

\begin{cor}
	
 Under conditions of Theorem~\ref{thm: oscillators complicated}, system \eqref{eq: main oscillators system with friction} satisfies Assumption~\ref{asum1} with $W=H$, $r=3\cdot 2^{N+1}-5$ and $Q$ such that $\cK^H_Q \supset K$. Hence it admits a strict Lyapunov function $W^\sharp$ outside a compact set, constructed in Proposition~\ref{p:strict_Lyap}.
	
\end{cor}

%
%

 Proof of Theorem~\ref{thm: oscillators complicated} is based on the following three lemmata.

\begin{defn}
	Let $g:\mathbb{R}^{2N}\to\mathbb{R}$ be a smooth function. We say that $g$ has order $k$ at a point $(p,q)$ (w.r.t\ system~\eqref{eq: main oscillators system with friction}) and write $\ord(g)(p,q)=k$ if $L_F^j g(p,q)=0$ for all $0\le j\le k-1$ and $L_F^{k}g(p,q)\ne 0$. Sometimes we omit $(p,q)$ and write $\ord(g)$ for brevity.
	If $L^j_F g(p,q)=0$ for all $j$, we write $\ord(g)(p,q)=\infty.$
\end{defn}

\begin{lem}
	\begin{itemize}
		\item if $\ord(g)>0$, then $\ord(L_F g)=\ord(g)-1$;
		\item $\ord(g_1g_2)=\ord(g_1)+\ord(g_2)  $;
		\item $\ord(g_1+g_2)\ge\min\{\ord(g_1),\ord(g_2)\}$;
		\item if $\ord(f)\ne\ord(g)$, then $\ord(f+g)=\min\{\ord(f),\ord(g)\}$.
	\end{itemize}
	
\end{lem}

\begin{proof}
	Follows from the fact that $L_F$ is a first order differential operator.
\end{proof}

\begin{lem}
	\label{lm: orders}
	
	If $\ord(H)\ge 3\cdot 2^{N+1}-5$ at a point $(p,q)$, then $p=0$ and $q$ satisfies system~\eqref{eq: main equilibrium system}.
	
\end{lem}

\begin{proof}
	We again begin with the identity
	\[
	L_F^{k+1} H = -\sum_{j=0}^k \binom{k}{j} L_F^jp_1 L_F^{k-j}p_1,
	\]
	which implies
	\[
	\ord(p_1)= \frac12  \ord(L_F H)   \ge 3(2^N-1).
	\]
	Let us compute orders of $p_k$ for $k\ge 2$. For $p_1$ we obtain
	\[
	L_F p_1 = -U_1'-V_1'=0
	\quad\text{and}\quad
	L_F^2 p_1 = -(U_1''+V_1'')p_1 + V_1''p_2.
	\]	
	Since $\ord(p_1) = \ord(L_F^2p_1)+2$, we obtain
	\[
	\ord(V_1''p_2) = \ord(L_F^2 p_1 + (U_1''+V_1'')p_1) = \ord(L_F^2p_1).
	\]
	Hence $\ord(p_2)\le \ord(p_1)-2$. If $V_1''\ne 0$, then $\ord(p_2)=\ord(p_1)-2$. On the other case $V_1''=0$, we have $V_1'''\ne 0$ by the assumption of the theorem and $\ord(p_1)-2=\ord(p_2)+\ord(V_1'')$. Since $L_FV_1''=V_1'''(p_1-p_2)$, we have $\ord(V_1'')=\ord(p_2)+1$. So $\ord(p_1)-2=2\ord(p_2)+1$ and $\ord(p_2)=\frac12(\ord(p_1)-3)$. Now we unite the both cases. Since $\frac12(j-3)\le j-2$ for any $j\ge 1$ and $\ord(p_1)\ge 1$, we obtain
	\[
	\ord(p_1)-2\ge \ord(p_2)\ge \frac12(\ord(p_1)-3)\ge 3(2^{N-1}-1).
	\]
	
	Similarly,
	\[
	L_F p_2 = -U_2'+V_1'-V_2'
	\quad\text{and}\quad
	L_F^2 p_2 = V_1''p_1-(U_2''+V_2''+V_1'')p_2 + V_2''p_3
	\]
	and
	\[
	\ord(p_2)-2\ge\ord(p_3)\ge 3(2^{N-2}-1).
	\]
	
	Proceeding by induction, we obtain $\ord(p_k)\ge 3(2^{N-k+1}-1)$. In particular, since $k\le N$, we have $\ord(p_k)\ge 3$. Hence $p_k=L_F p_k=0$ for all $k$. Therefore, $p=0$ and $q$ is a solution to system~\eqref{eq: main equilibrium system}.
\end{proof}

\begin{lem}
	\label{lm: set of zeros is bounded}	
	Set of solutions to system~\eqref{eq: main equilibrium system} is bounded. 
	
\end{lem}

\begin{proof}  

		Let
	\[
	M:= \max_{|x|\le R}\max_{1\le k \le N} |V'_k(x)| + 1 \ge 1.
	\]
	Since $ V_k''(x) > 0$ for $x\le -R$, 
	\be\lbl{x<R}
	x\le R \qu\Rightarrow \qu V_k'(x) \le M - 1 \qu\forall k.
	\ee
	Similarly, since $ V_k''(x) > 0$ for $x\ge R$,
	\be\lbl{x>-R}
	x \ge -R \qu\Rightarrow \qu V_k'(x) \ge -M + 1 \qu\forall k.
	\ee
	Since $U_k'(x)\to\infty$ as $x\to \infty$, there is $a> R$ such that
	\be\lbl{a_VU}
	x \ge a \qu \Rightarrow \qu U'_k(x) \ge 2M\qu \forall k.
	\ee
We claim that any solution $q$ to system \eqref{eq: main equilibrium system} satisfies $q_k < a$ for all $k$.
Indeed, let $k$ be the smallest index for which $q_k \ge a$, so that in particular $q_{k-1}<q_k$ (if $k\ne 1$).
Then, 
$
U_k'(q_k) \ge 2M
$
due to \eqref{a_VU},
and  $V'_{k-1}(q_{k-1} - q_k) \le M - 1$, due to \eqref{x<R}, where we set $V'_{0}:=0$.
Then,  
\[
U_k'(q_k) - V'_{k-1}(q_{k-1} - q_k) \ge M + 1.
\]
In view of the $k$-th equation from~\eqref{eq: main equilibrium system},  this implies
$V'_{k}(q_{k} - q_{k+1}) \le - M - 1$.
Accordingly, by \eqref{x>-R}, $q_k - q_{k+1} < -R$, so $q_{k+1} > q_k \ge a$.
Continuing by induction, $N-1$-st equation in~\eqref{eq: main equilibrium system} gives $q_N > q_{N-1} \ge a$, so  
\[U'_N(q_N) - V_{N-1}'(q_{N-1} - q_N) \ge 2M - (M-1) \ge 2\] 
which contradicts to the last equation in~\eqref{eq: main equilibrium system}.

Similarly we show that every solution $q$ is bounded from below by $-b<-R$ such that 
\be\lbl{b_VU}
x \le -b \qu \Rightarrow \qu U'_k(x) \le -2M\qu \forall k.
\ee

\end{proof}

\begin{proof}[Proof of Theorem~\ref{thm: oscillators complicated}]
	
If  $L_F^jH(p,q)=0$ for all $1\le j\le 3\cdot 2^{N+1}-5$ at a point $(p,q)$, then $p=0$ and $q$ satisfies~\eqref{eq: main equilibrium system} by  Lemma~\ref{lm: orders}. Hence $q$ is bounded by Lemma~\ref{lm: set of zeros is bounded}.
	
\end{proof}

\section{Mixing in chains coupled to Langevin thermostats}
\lbl{s:mixing}

Let us consider either the chain of rotators with Hamiltonian \eqref{H_rot} or the chain of  oscillators with Hamiltonian~\eqref{H_osc}.
 We couple the first and the last particles in the chain with Langevin thermostats of positive temperatures $T_1$ and $T_N$ correspondingly. That is, we consider the system of stochastic differential equations
\be\lbl{eq_stoch}
\dot q_j = p_j, \qquad \dot p_j = -\p_{q_j} H +(\de_{1j} + \de_{Nj}) (- p_j + \sqrt{2T_j}\dot\beta_j(t)), \qu 1\le j \le N,  
\ee
where $\beta_j(t)$ are standard independent Brownian motions. The initial conditions $(p_0,q_0)$ are assumed to be random and independent from the Brownian motions $\beta_j$.

Denote by $\cM\ni(p,q)$ the phase space of the system, i.e. $\cM = \R^N\times \T^N$ in the case of rotators and $\cM = \R^{2N}$ in the case of oscillators. Let $\mu$ be a probability measure on $\cM$ and
$(p,q)(t)$ be a solution to \eqref{eq_stoch}  with initial conditions $(p_0,q_0)$, distributed accordingly to the measure $\mu$. 
If the distribution of solution at time $t$, denoted by $\cD(p,q)(t)$, satisfies 
\[
\cD(p,q)(t) \equiv \mu \qmb{for all}\qu t\ge 0, 
\]
then the measure $\mu$ is called  {\it stationary}.
If equation \eqref{eq_stoch} has a unique stationary measure $\mu$ and distribution $\cD(p,q)(t)$ of {\it any}
\footnote{With "reasonable"\ initial conditions, e.g. with finite second moment.}
 its solution $(p,q)(t)$ weakly converges to $\mu$ as $t\to\infty$, then  the equation is called {\it mixing}.

Establishing the mixing property for equation \eqref{eq_stoch} is a problem of prime interest in non-equilibrium statistical mechanics of solids. 
Let us first consider the case of oscillators and assume that the potentials $V_j$ and $U_j$ from \eqref{H_osc}  behave at infinity polynomially. Then, due to our knowledge,  for chains of arbitrary length $N$ the mixing property is proven only in the case when the interaction potentials $V_j$ are "not weaker" than the pinning potentials $U_j$, i.e. $V_j$ behave at infinity as  polynomials of the same or higher degree than those associated with the potentials $U_j$ \cite{EPRB, EH, RBT, C, CEHRB}. In the opposite case when the pinning dominates, the mixing is proven only for a chain of $N=3$ oscillators \cite{HM}.

As for the chain of rotators,
in \cite{CEP_3, CE_4} the authors argue that it can be viewed as a system of oscillators with extremely strong pinning potentials. They establish the mixing property for the chains of length $N=3,4$  while for longer chains the problem remains open.
The reason why it is so hard to prove the  mixing in the chains of rotators or oscillators with "strong" pinning, is in the discussed in the Introduction decoupling of "fast" rotators (or oscillators) from their "slow" neighbours, leading to extremely slow energy transport through the chain. 
In particular, as a result the rate of mixing  turns our to be slower than exponential, in contrast to the case of oscillators with "strong" interaction, see \cite{CP} additionally to the above mentioned works.

Let us note that for chains  in which {\it each} particle interacts with its own Langevin thermostat,  the mixing property is well-understood.

One of the most effective approaches to problems of this kind relies on the fact that to establish the mixing property it suffices to construct a strict Lyapunov function $W:\cM\mapsto\R$ outside a compact set for the generator  $\Lc^{T_1,T_2}$ of equation  \eqref{eq_stoch}. More specifically, it suffices to find a sufficiently smooth proper function $W$, satisfying 
\[
\Lc^{T_1,T_2} W \le -\psi\circ W + C,
\] 
where $C$ is a constant and $\psi$ is a positive increasing function,  see for details Theorem~2.1 in \cite{CEP_3}   and references therein.
We recall that action of the generator $\Lc^{T_1,T_2}$ on a $C^2$-smooth function $f(p,q)$ can be computed as
\[
\Lc^{T_1,T_2} f = L_F f + T_1\frac{\p^2 f}{\p p_1^2} + T_N\frac{\p^2 f}{\p p_N^2},
\]
where $L_F$ stands for the Lie derivative along the vector field given by the deterministic part of equation \eqref{eq_stoch}.

 For the chain of oscillators with "strong" interaction which dominates the pinning,  the mixing in \cite{C, CEHRB} is established by proving that $\exp(\theta H)$ gives the desired Lyapunov function once $\theta$ is small enough. In the opposite case as well as in the case of rotators the employed arguments do not work. Instead, in \cite{CEP_3, CE_4} the authors construct a suitable Lyapunov functions for the chains of rotators of lengths $N=3$ and $4$ using a method related to the discussed in the Introduction KAM-like approach subsequently employed in \cite{CEW}. For the moment of writing it did not allow to prove the mixing property for longer chains since for them it is developed only for specific regime when the  most part of energy is concentrated in a single rotator. 

Now let us consider equation similar to \eqref{eq_stoch}, in which the Langevin thermostat is coupled only to the first oscillator, i.e. in \eqref{eq_stoch} the factor $\de_{1j}+\de_{Nj}$ is replaced by $\de_{1j}$. Its generator $\Lc^{T_1}$ acts to a $C^2$-smooth function $f(p,q)$ as
$
\ds{\Lc^{T_1} f = L_F f + T_1\frac{\p^2 f}{\p p_1^2}}.
$
 Then, being applied to the Lyapunov function $W$ constructed in the Theorem~\ref{t:rotators}, it satisfies
\be\lbl{Gen_W}
\Lc^{T_1} W \le -H + C + T_1\frac{\p^2 W}{\p p_1^2}.
\ee
By definition \eqref{LF_rot} of $W$, 
\[\frac{\p^2 W}{\p p_1^2}\sim a_0\ga_0H^{\ga_0 - 1} = a_0\ga_0 H^{2N-2} 
\qmb{for}\qu H\gg 1.
\]
 Then, already for $N=2$ the r.h.s. of \eqref{Gen_W} growth as $H^2$, so $W$ is {\it not} a Lyapunov function for the generator $\Lc^{T_1}$.
Similarly one can see 
that our approach does not allow to prove mixing for chain \eqref{eq_stoch} even for $N=3$, in contrast to  the method suggested in \cite{CEP_3, CE_4, CEW}.
However, our construction is very explicit, simple and universal, and it allows to get the rate of energy dissipation seemingly achieved in the worst scenario.
We hope that its suitable modification could help to construct the desired Lyapunov functions for generators of systems \eqref{eq_stoch}.

\medskip

\noindent{\bf Funding.}     
This work was supported by
the Russian Science Foundation under grant no. 19-71-30012, https://rscf.ru/en/project/23-71-33002/. 

\end{document}